\documentclass[10pt,leqno]{amsart}
\usepackage[text={32cc,45cc},centering]{geometry}

\usepackage{amssymb,amsthm,amsmath,amsfonts,gensymb,mathrsfs}
\usepackage{mathtools,enumitem,theoremref}
\usepackage{hyperref}
\newtheorem{thm}{Theorem}
\newtheorem{lem}[thm]{Lemma}

\usepackage[backend=biber,style=ieee-alphabetic,maxnames=99]{biblatex}
\newcommand{\T}{\mathscr{T}}
\DeclareMathOperator{\dom}{\mathrm{dom}}
\newcommand{\restr}{\!\upharpoonright\!}
\newcommand{\abs}[1]{\left\lvert{#1}\right\rvert}
\newcommand{\paren}[1]{\left({#1}\right)}
\newcommand{\bdef}[1]{\emph{#1}}

\hypersetup{colorlinks=false}

\begin{document}
\title{A note on the indivisibility of the Henson graphs}
\author[K.\ Gill]{Kenneth Gill}
\date{October 30, 2023}
\email{gillmathpsu@posteo.net}
\maketitle

{\let\thefootnote\relax
\footnotetext{This work is part of the author's Ph.D.\ dissertation at
  Penn State University \cite{mythesis}.} 

\begin{abstract}
  We show that in contrast to the Rado graph, the Henson graphs are not
  computably indivisible.
\end{abstract}

\bigskip

\section{Introduction}

The Rado graph is, up to graph isomorphism, the unique countable undirected
graph that satisfies the following property: if $A$ and $B$ are any finite
disjoint sets of vertices, there is a vertex not in $A$ or $B$ which is
connected to every member of $A$ and to no member of $B$. It is homogeneous and
universal for the class of finite graphs.

Our interest here lies with the closely related family of \bdef{Henson graphs},
introduced by C.\ Ward Henson in 1971 \cite{Hen71}. For each $n\geq3$, the
Henson graph $H_n$ is up to isomorphism the unique countable graph which
satisfies the following property analogous to that characterizing the Rado
graph: for any finite disjoint sets of vertices $A$ and $B$, if $A$ does not
contain a copy of $K_{n-1}$, then there is a vertex $x\notin A\cup B$ connected
to every member of $A$ and to no member of $B$. (Here we write $K_m$ for the
complete graph on $m$ vertices.) The graph $H_n$ is homogeneous and universal for
the class of $K_n$-free finite graphs.

We presume familiarity with the basic terminology of computable structure
theory, as for example in the first chapter of \cite{Mon1}. A structure
$\mathcal{S}$ is said to be \bdef{indivisible} if for any presentation
$\mathcal{A}$ of $\mathcal{S}$ and any coloring $c$ of $\dom \mathcal{A}$ with
finite range, there is a monochromatic subset of $\dom \mathcal{A}$ which
induces a substructure isomorphic to $\mathcal{S}$. We call the monochromatic
subset in question a \bdef{homogeneous set} for $c$. $\mathcal{S}$ is
\bdef{computably indivisible} if there is a homogeneous set computable from
$\mathcal{A}$ and $c$, for any presentation $\mathcal{A}$ and coloring $c$ of
$\dom \mathcal{A}$.

For the rest of the paper, we fix a computable presentation of $H_n$ with domain
$\mathbb{N}$ and thus focus only on the coloring. Viewed as a structure in the
language of a single binary relation, the Rado graph is known to be indivisible,
and computably so (folklore). Each of the Henson graphs is also indivisible.
Henson himself proved that a weak form of indivisibility holds for each $H_n$.
Full indivisibility was first shown for $n=3$ by Komj\'ath and R\"odl
\cite{KR86}, and then for all $n$ by El-Zahar and Sauer \cite{ElZS89}. (A
clarified and corrected version of the proof of Komj\'ath and R\"odl can be
found in \cite{mythesis}.) Work on the Ramsey theory of the Henson graphs has
progressed beyond vertex colorings; recently, Natasha Dobrinen has undertaken a
deep study of the structure of $H_n$ and shown that for each $n$, $H_n$ has
finite big Ramsey degrees, developing many novel techniques in the process
\cite{DobH3,DobHn}.

Our far more modest result concerns only vertex colorings and states that
unlike the Rado graph, none of the Henson graphs is computably indivisible:

\begin{thm}\thlabel{Hn-noce} For every $n\geq3$, there is a computable 2-coloring
  of $H_n$ with no c.e.\ homogeneous set.
\end{thm}

This theorem naturally raises the question of how complicated a homogeneous set
for a coloring of $H_n$ can or must be. An analysis of the proof of the
indivisibility of $H_3$ by Komj\'ath and R\"odl in \cite{KR86} demonstrates that
a homogeneous set can always be computed in the first jump of the coloring. For
$H_n$ in general, the proof of El-Zahar and Sauer in \cite{ElZS89} shows that
the $(2n-3)$rd jump of a coloring suffices to compute a homogeneous set. The
latter is a strictly worse upper bound for $n=3$, and it is currently unknown
whether a similar discrepancy exists for any $n\geq 4$. Where vertex colorings
of $H_n$ fall on the spectrum of coding vs.\ cone avoidance is another
intriguing question.

%%%%%%%%%%%%%%%%%%%%%%%%%%%%%%%%%%%%%%%%%%%%%%%%%%%%%%%%%%%%%%%%%%%%%%%%%%%%%%%

\section{Proof of the theorem}

Write $x\in G$, for a graph $G$, to mean $x$ is a vertex of $G$. By abuse of
notation, if $V\subset G$ is any set of vertices, we will identify $V$ with the
induced subgraph of $G$ on $V$. Furthermore, we always identify natural numbers
with the elements of $H_n$ they encode via our fixed computable presentation of
$H_n$, and sets of naturals with the corresponding induced subgraphs of $H_n$.
If $A=\{a_1<\cdots <a_n\}$ and $B=\{b_1<\cdots < b_n\}$ are two sets of vertices
in a graph $G$, write $A\simeq^\ast B$ if the map $a_i\mapsto b_i$ is an
isomorphism of induced subgraphs. If the vertices of $G$ are given some linear
ordering, denote by $G \restr m$ the induced subgraph of $G$ on its first $m$
vertices. If $x\in G$, let $G(x)$ denote the induced subgraph of $G$ consisting
of the neighbors of $x$. A set of the form $G(x)$ is referred to as a ``neighbor
set''. Let $\T_n$ be the set of finite $K_n$-free simple connected graphs. 

We will need two lemmas. The first is a consequence of the following theorem of
Jon Folkman, which appears as Theorem~2 in \cite{F70}. For a graph $G$, let
$\delta(G)$ be the largest $n$ such that $G$ contains a subgraph isomorphic to
$K_n$.

\begin{thm}[Folkman]\thlabel{Folkman} For each $k>0$ and finite graph $F$, there
  is a finite graph $G$ such that
  \begin{enumerate}[label=(\alph*)]
    \item $\delta(G) = \delta(F)$, and
    \item for any partition of the vertices of $G$ as $G_1 \sqcup \cdots \sqcup
      G_k$, there is an $i$ such that $G_i$ contains a subgraph isomorphic to
      $F$.
  \end{enumerate}
\end{thm}
Part (a) implies that $G$ is $K_n$-free if $F$ is.

\begin{lem}\thlabel{fact1} For each $n$ and $k$, there is a $G\in\T_n$ which is
  not an induced subgraph of $\bigcup_{i=1}^k H_n(x_i)$ for any vertices $x_1,
  \dotsc,x_k\in H_n$. In particular, no finite union of neighbor sets in $H_n$
  can contain an isomorphic copy of $H_n$.
\end{lem}
\begin{proof} By applying \thref{Folkman} with $F=K_{n-1}$, there is a
  $K_n$-free $G$ such that for every partition of $G$ into $k$ sets, at least
  one set contains a $K_{n-1}$. Since a neighbor set in $H_n$ cannot contain a
  $K_{n-1}$, this means that $G$ is not contained in any union of $k$ neighbor
  sets. \qedhere
\end{proof}
Note that the graph $G$ can be found computably from $n$ and $k$ by a brute-force
search. The next fact is a restatement of Lemma~1 of \cite{ElZS89}:

\begin{lem}[El-Zahar \& Sauer]\thlabel{fact2} Let $\Delta$ be a finite induced
  subgraph of $H_n$ with $d$ vertices. Let $\Gamma$ be any member of $\T_n$ with
  $d+1$ vertices put in increasing order such that $\Delta \simeq^\ast
  \Gamma\restr d$. Then there are infinitely many choices of $x \in H_n$ such
  that $\Delta \cup \{x\} \simeq^\ast \Gamma$.
\end{lem}

\begin{proof}[Proof of \thref{Hn-noce}] The proof is by a finite injury priority
  argument. We build a computable $c\colon H_n\to 2$, viewing $2$ as the set
  $\{R,B\}$ (red and blue), to meet requirements
  \begin{align*}  
    R_e\colon \paren{ \abs{W_e}=\infty \land \abs{c(W_e)}=1} \implies
    \text{\thref{fact2} fails if $H_n$ is replaced with $W_e \subset H_n$}.
  \end{align*}
  These are given the priority order $R_0>R_1>R_2>\cdots$. We also define a
  computable function $p$ in stages, where $p(x,s)$ is the planned color of
  vertex $x$ at stage $s$, beginning with $p(x,0)=R$ (red) for all $x$. This
  function will be used to keep track of vertices which requirements ``reserve''
  to be a certain color. Only one vertex will actually be colored at each stage,
  starting with $c(0)=R$.

  A requirement $R_e$ is said to be active at stage $s$ if $e\leq s$ and
  $W_{e,s}$ contains at least one element that was enumerated after the most
  recent stage in which $R_e$ was injured (to be explained below). If $R_e$ was
  never injured, we say it is active simply if $W_{e,s}\neq \emptyset$. Each
  requirement $R_e$ will amass a finite list of vertices $\{x_e^1, x_e^2,
  \dotsc, x_e^\ell\}$ in $W_e$ as its followers, together with a target graph
  $\Gamma_e$ (also explained below). When a follower $x_e^m$ is added, $R_e$
  will set the function $p(x,s)$ for some vertices $x\in H_n(x_e^m)$; we say
  $R_e$ \bdef{reserves} $x$ when it sets $p(x,s)$. Weaker requirements cannot
  reserve vertices which are currently reserved by stronger requirements. The
  followers, target graph, and reservations of $R_e$ are canceled when $R_e$ is
  injured by a stronger requirement. (Canceling a reserved vertex just means the
  vertex is no longer considered to be reserved by $R_e$, and does not change
  the values of $p$ or $c$.) We may as well assume each $W_e$ is monochromatic,
  and will refer to $R_e$ as either a red or blue requirement accordingly.

  We now detail the construction, and afterwards show that all requirements
  are injured at most finitely often and are met. First, if no $R_e$ is active
  at stage $s+1$ for $e\leq s$, set $p(x,s+1) = p(x,s)$ for all $x$, set
  $c(s+1)=p(s+1,s+1)$, and end the stage. If a requirement $R_e$ is already
  active at stage $s+1$ and has no follower, give it a follower $x_e^1$ which is
  any element of $W_{e,s}$ that was enumerated after the stage in which $R_e$
  was last injured, or otherwise any element of $W_{e,s}$ if $R_e$ was never
  injured. Then for every $y\in H_n(x_e^1)$ which is not currently reserved by a
  stronger requirement and has not yet been colored, reserve $y$ by setting
  $p(y,s+1)$ to be the opposite color as $c(x_e^1)$.

  If $R_e$ is active and has a follower at stage $s+1$ but no target graph, let
  its target graph be some $\Gamma_e\in \T_n$ which cannot be contained in $k+1$
  neighbor sets, where $k$ is the total number of all followers of stronger
  currently active requirements. Such a $\Gamma_e$ may be furnished by
  \thref{fact1}. Order $\Gamma_e$ in such a way that each vertex (except the
  first) is connected to at least one previous vertex. 

  Next, suppose that at least one requirement is active and has a follower and
  target graph at stage $s+1$. Go through the following procedure for each such
  $R_e$ in order from strongest to weakest. Let $m$ be the number of followers
  of $R_e$ at stage $s$; we will at this point have $\{x_e^1,\dotsc, x_e^m\}
  \simeq^\ast \Gamma_e \restr m$. Suppose there is some $x\in W_{e,s+1}$ with
  $x$ greater than the stage at which $x_e^m$ was enumerated into $W_e$, and
  such that $\{x_e^1,\dotsc, x_e^m, x\} \simeq^\ast \Gamma_e\restr (m+1)$. If
  so, then give $R_e$ the new follower $x_e^{m+1} = x$, and for all $y\in
  H_n(x_e^{m+1})$ with $y>s+1$ such that $y$ is not currently reserved by any
  stronger requirement, have $R_e$ reserve $p(y,s+1) = R$ if $R_e$ is blue, or
  $p(y,s+1)=B$ if $R_e$ is red. Injure all weaker requirements by canceling
  their followers, target graphs, and reservations. After this is done for all
  active $R_e$, end the stage by making $p(z,s+1)=p(z,s)$ for any $z$ for which
  $p(z,\cdot)$ was not modified earlier in the stage, and then letting
  $c(s+1)=p(s+1,s+1)$. If instead no $x$ as above was found for any active
  $R_e$, then set $p(x,s+1) = p(x,s)$ for all $x$, set $c(s+1)=p(s+1,s+1)$, and
  end the stage. This completes the construction.

  Each requirement only need accumulate a finite list of followers, so in
  particular $R_0$ will only injure other requirements finitely many times.
  After the last time a requirement is injured, it only injures weaker
  requirements finitely often, so inductively we have that every requirement is
  only injured finitely many times before acquiring its final list of followers
  and target graph. And each requirement is satisfied: suppose (without loss of
  generality) $R_e$ is blue. For each $i\geq 2$, the vertex $x_e^i$ is an
  element of $H_n(x_e^j)$ for some $j<i$, by assumption on how we have ordered
  $\Gamma_e$. If $x_e^j$ was enumerated into $W_e$ at stage $s$, then when this
  $x_e^j$ was chosen as a follower, $R_e$ reserved every element of $H_n(x_e^j)$
  greater than $s$ by making its planned color red---except for those vertices
  which were already reserved (to be blue) by stronger (red) requirements.
  Therefore, if $x_e^i$ is blue, then since in particular the construction
  requires $x_e^i > s$, we must have $x_e^i$ a neighbor of some follower of a
  stronger (red) requirement. (We asked for $x_e^i$ to be greater than the stage
  $t$ at which $x_e^{i-1}$ was enumerated. Such an $x_e^i$ can be found for any
  $t$ by \thref{fact2}.) So this copy we are building of $\Gamma_e$ inside $W_e$
  is contained entirely in a union of neighbor sets of followers of stronger
  active requirements, except possibly for $x_e^1$ which may lie outside of any
  such neighbor set. If $R_e$ is never injured again, then the number $k$ of
  such followers never changes again; it is the same as it was when the target
  graph $\Gamma_e$ was chosen not to fit inside $k+1$ neighbor sets. The latter
  number is large enough to also cover $x_e^1$, so that this copy of $\Gamma_e$
  can never be completed inside $W_e$, implying \thref{fact2} fails in $W_e$.
  \qedhere
  
\end{proof}

\noindent\textbf{Acknowledgements:} This research was supported in part by NSF
grant DMS-1854107. I am extremely grateful to my thesis advisors Linda Brown
Westrick and Jan Reimann for their invaluable help, and also to Peter Cholak for
his comments on an earlier version of the proof of \thref{Hn-noce}.

\printbibliography

\end{document}